\documentclass[a4paper]{amsart}
\usepackage[T1]{fontenc}
\usepackage[british]{babel}
\usepackage{mathtools}
\usepackage[hidelinks]{hyperref}
\DeclareSymbolFont{bbold}{U}{bbold}{m}{n}
\DeclareSymbolFontAlphabet{\mathbbm}{bbold}
\DeclareMathSymbol{\rest}{\mathbin}{AMSa}{"16}
\newcommand{\B}{\mathbb{B}}
\newcommand{\C}{\mathbb{C}}
\newcommand{\M}{\mathcal{M}}
\newcommand{\N}{\mathcal{N}}
\newcommand{\nwd}{\mathrm{nwd}}
\newcommand{\bbot}{\mathbbm{0}}
\newcommand{\btop}{\mathbbm{1}}
\newcommand{\bb}{\mathfrak{b}}
\newcommand{\dd}{\mathfrak{d}}
\newcommand{\rr}{\mathfrak{r}}
\newcommand{\sss}{\mathfrak{s}}
\newcommand{\uu}{\mathfrak{u}}
\DeclareMathOperator{\add}{add}
\DeclareMathOperator{\cof}{cof}
\DeclareMathOperator{\non}{non}
\DeclareMathOperator{\Part}{Part}
\DeclareMathOperator{\Partb}{\mathbf{Part}}
\DeclareMathOperator{\Dense}{Dense}
\DeclareMathOperator{\suc}{Seq}
\DeclarePairedDelimiter{\abs}{\lvert}{\rvert}
\DeclarePairedDelimiterX{\Set}[2]{\{}{\}}{ #1 \mathchoice{\:}{\:}{\,}{\,}\delimsize\vert\allowbreak\mathchoice{\:}{\:}{\,}{\,}\mathopen{} #2 }
\DeclarePairedDelimiterX{\Seq}[2]{\langle}{\rangle}{ #1 \mathchoice{\:}{\:}{\,}{\,}\delimsize\vert\allowbreak\mathchoice{\:}{\:}{\,}{\,}\mathopen{} #2 }
\DeclarePairedDelimiterXPP{\eq}[2]{}{[}{]}{_{#2}}{#1}
\theoremstyle{plain}
\newtheorem{theorem}{Theorem}[section]
\newtheorem{proposition}[theorem]{Proposition}
\newtheorem{lemma}[theorem]{Lemma}
\newtheorem{corollary}[theorem]{Corollary}
\newtheorem{claim}{Claim}
\theoremstyle{definition}
\newtheorem{definition}[theorem]{Definition}
\newtheorem{question}[theorem]{Question}
\theoremstyle{remark}
\newtheorem{remark}[theorem]{Remark}
\begin{document}
\title{Almost refinement, reaping, and ultrafilter numbers}
\author{J{\"o}rg Brendle}
\address{Graduate School of System Informatics\\Kobe University\\1-1 Rokkodai-cho Nada-ku\\Kobe 657-8501\\Japan}
\author{Michael Hru{\v s}{\'a}k}
\address{Centro de Ciencias Matem{\'a}ticas\\Universidad Nacional Aut{\'o}noma de M{\'e}xico\\Morelia 58089\\Mexico}
\author{Francesco Parente}
\address{Graduate School of System Informatics\\Kobe University\\1-1 Rokkodai-cho Nada-ku\\Kobe 657-8501\\Japan}
\thanks{The second author gratefully acknowledges support received from the PAPIIT-UNAM grant IN 101323, and the hospitality of the Kobe University. The third author is an International Research Fellow of the Japan Society for the Promotion of Science}
\begin{abstract} We investigate the combinatorial structure of the set of maximal antichains in a Boolean algebra ordered by almost refinement. We also consider the reaping relation and its associated cardinal invariants, focusing in particular on reduced powers of Boolean algebras. As an application, we obtain that, on the one hand, the ultrafilter number of the Cohen algebra is greater than or equal to the cofinality of the meagre ideal and, on the other hand, a suitable parametrized diamond principle implies that the ultrafilter number of the Cohen algebra is equal to $\aleph_1$.
\end{abstract}

\maketitle

\section{Introduction}

Let $A$ and $B$ be partitions of $\omega$; let us say that $B$ \emph{almost refines} $A$, in symbols $A\le^* B$, if all but finitely many blocks of $A$ are a union of blocks of $B$. The relation $\le^*$ was first considered by Matet~\cite{MR0830067} and further investigated in the literature \cite{MR1771096,MR1489310,MR1601919,MR1755813}, where several ``dual'' cardinal invariants have been introduced and studied.

In the first part of the paper, we aim to generalize this analysis from partitions of $\omega$ to maximal antichains in c.c.c.\ Boolean algebras. More precisely, we equip the set $\Part(\B)$ of all maximal antichains in a c.c.c.\ Boolean algebra $\B$ with the relation $\le^*$, defined analogously as for partitions of $\omega$. We then use generalized Galois-Tukey connections to study the relational system $\Partb^*(\B)=\langle\Part(\B),\le^*,\Part(\B)\rangle$ and its associated cardinal invariants.

After some preliminaries on relational systems, which are collected in Section~\ref{section:due}, we prove some basic results about $\Partb^*(\B)$ in Section~\ref{section:tre}. Among other things, we show that, if $\B$ is a non-atomic $\sigma$-finite c.c.\ Boolean algebra, then there exists a generalized Galois-Tukey connection from the dominating relation $\langle\prescript{\omega}{}{\omega},\le^*,\prescript{\omega}{}{\omega}\rangle$ to $\Partb^*(\B)$. Subsequently, we focus on the Cohen algebra $\C_\omega$ and prove that $\Partb^*(\C_\omega)$ is Galois-Tukey equivalent to the nowhere dense ideal.

In the second part of this paper, we introduce the reaping relation of a Boolean algebra $\B$, whose associated cardinal invariants are the well-known reaping and splitting numbers of $\B$. In Section~\ref{section:quattro}, in particular, we show that the reaping relation of the reduced power of a Boolean algebra $\B$ is related both to $\Partb^*(\B)$ and to the reaping relation of $\mathcal{P}(\omega)/\mathrm{fin}$. Furthermore, we completely determine the reaping and splitting numbers of the reduced power of the Cohen algebra.

Finally, in Section~\ref{section:cinque} we apply the results established in the previous section to derive consequences for the ultrafilter number of Boolean algebras. As a result, we obtain that $\cof(\M)\le\uu(\C_\omega)$; this stands in parallel with the work of Burke~\cite{MR0961402} which, for the random algebra $\B_\omega$, gives the lower bound $\cof(\N)\le\uu(\B_\omega)$. In addition to that, we extend the technique of parametrized diamond principles, due to Moore, Hru{\v s}{\'a}k, and D{\v z}amonja~\cite{MR2048518}, to a class of Boolean algebras which we call ``Borel-homogeneous''. This yields that, if the parametrized diamond principle related to the reaping relation of the reduced power of $\C_\omega$ holds, then $\uu(\C_\omega)=\aleph_1$.

\section{Relational systems and generalized Galois-Tukey connections}\label{section:due}

In this section, we recall the basic notions and results regarding generalized Galois-Tukey connections, first introduced by Vojt{\'a}{\v s}~\cite{MR1234291} and later systematized by Blass~\cite{MR2768685}. In what follows, triples $\mathbf{A}=\langle A_-,A,A_+\rangle$, with $A\subseteq A_-\times A_+$, are called \emph{relational systems}.

\begin{definition} Let $\mathbf{A}$ be a relational system. If for every $x\in A_-$ there exists $y\in A_+$ such that $\langle x,y\rangle\in A$, then we define the \emph{dominating number} of $\mathbf{A}$ as
\[
\dd(\mathbf{A})=\min\Set{\abs{Y}}{Y\subseteq A_+\text{ and }(\forall x\in A_-)(\exists y\in Y)(\langle x,y\rangle\in A)}.
\]

If for every $y\in A_+$ there exists $x\in A_-$ such that $\langle x,y\rangle\notin A$, then we define the \emph{bounding number} of $\mathbf{A}$ as
\[
\bb(\mathbf{A})=\min\Set{\abs{X}}{X\subseteq A_-\text{ and }(\forall y\in A_+)(\exists x\in X)(\langle x,y\rangle\notin A)}.
\]
\end{definition}

\begin{definition} A \emph{generalized Galois-Tukey connection} from a relational system $\mathbf{A}$ to a relational system $\mathbf{B}$ consists of two functions $\varphi_-\colon A_-\to B_-$ and $\varphi_+\colon B_+\to A_+$ such that for all $a\in A_-$ and all $b\in B_+$
\[
\langle\varphi_-(a),b\rangle\in B\implies\langle a,\varphi_+(b)\rangle\in A.
\]
If there exists a generalized Galois-Tukey connection from $\mathbf{A}$ to $\mathbf{B}$, then we shall write $\mathbf{A}\le_\mathrm{T}\mathbf{B}$. Moreover, $\mathbf{A}$ and $\mathbf{B}$ are \emph{Galois-Tukey equivalent}, in symbols $\mathbf{A}\equiv_\mathrm{T}\mathbf{B}$, if $\mathbf{A}\le_\mathrm{T}\mathbf{B}$ and $\mathbf{B}\le_\mathrm{T}\mathbf{A}$.
\end{definition}

Due to the following observation, generalized Galois-Tukey connections yield inequalities between the corresponding dominating and bounding numbers.

\begin{remark}[{Vojt{\'a}{\v s}~\cite[Observation 3.1.2]{MR1234291}}]\label{remark:vojtas} If $\mathbf{A}\le_\mathrm{T}\mathbf{B}$, then $\dd(\mathbf{A})\le\dd(\mathbf{B})$ and $\bb(\mathbf{B})\le\bb(\mathbf{A})$.
\end{remark}

Given a relational system $\mathbf{A}$, we denote
\[
\mathbf{A}^\sigma=\langle A_-,A^\sigma,\prescript{\omega}{}{A_+}\rangle,
\]
where $\langle a,f\rangle\in A^\sigma$ if and only if there exists $n<\omega$ such that $\langle a,f(n)\rangle\in A$. The following facts are straightforward.

\begin{proposition}\label{proposition:sigma} Let $\mathbf{A}$ and $\mathbf{B}$ be relational systems, then:
\begin{itemize}
\item $\mathbf{A}^\sigma\le_\mathrm{T}\mathbf{A}$;
\item $\dd(\mathbf{A})\le\dd(\mathbf{A}^\sigma)+\aleph_0$;
\item if $\mathbf{A}\le_\mathrm{T}\mathbf{B}$ then $\mathbf{A}^\sigma\le_\mathrm{T}\mathbf{B}^\sigma$.
\end{itemize}
\end{proposition}

We introduce an operation between relational systems which will be useful in Section~\ref{section:quattro}.

\begin{definition}[Blass~\cite{MR1356008}] Let $\mathbf{A}$ and $\mathbf{B}$ be relational systems; the \emph{sequential composition} of $\mathbf{A}$ and $\mathbf{B}$ is defined as
\[
\mathbf{A}\mathbin{;}\mathbf{B}=\bigl\langle A_-\times\prescript{A_+}{}{B_-},S,A_+\times B_+\bigr\rangle
\]
where
\[
\bigl\langle\langle x,f\rangle,\langle a,b\rangle\bigr\rangle\in S\iff\langle x,a\rangle\in A\text{ and }\langle f(a),b\rangle\in B.
\]
\end{definition}

By looking at the left and right projections of the Cartesian product, it is easy to verify that $\mathbf{A}\le_\mathrm{T}\mathbf{A}\mathbin{;}\mathbf{B}$ and $\mathbf{B}\le_\mathrm{T}\mathbf{A}\mathbin{;}\mathbf{B}$, respectively.

\begin{proposition}[{Blass~\cite[Proposition 2]{MR1367133}}] If $\mathbf{A}$ and $\mathbf{B}$ are relational systems, then $\dd(\mathbf{A}\mathbin{;}\mathbf{B})=\dd(\mathbf{A})\cdot\dd(\mathbf{B})$ and $\bb(\mathbf{A}\mathbin{;}\mathbf{B})=\min\{\bb(\mathbf{A}),\bb(\mathbf{B})\}$.
\end{proposition}

Lastly, let us recall the parametrized diamond principles of Moore, Hru{\v s}{\'a}k, and D{\v z}amonja~\cite{MR2048518}, which will play a role in Section~\ref{section:cinque}. For the purpose of the next definition, we stipulate that a relational system $\mathbf{A}$ is \emph{Borel} if
\begin{itemize}
\item $A_-$, $A$, and $A_+$ are Borel subsets of some Polish space;
\item for every $x\in A_-$ there exists $y\in A_+$ such that $\langle x,y\rangle\in A$;
\item for every $y\in A_+$ there exists $x\in A_-$ such that $\langle x,y\rangle\notin A$.
\end{itemize}

\begin{definition}[{Moore, Hru{\v s}{\'a}k, and D{\v z}amonja~\cite[Definition 4.4]{MR2048518}}]\label{def:diam} If $\mathbf{A}$ is a Borel relational system, then $\diamondsuit(\mathbf{A})$ is the following statement: for every Borel function $F\colon\prescript{<\omega_1}{}{2}\to A_-$ there exists $g\colon\omega_1\to A_+$ such that for every $f\colon\omega_1\to 2$ the set $\Set{\alpha<\omega_1}{\langle F(f\rest\alpha),g(\alpha)\rangle\in A}$ is stationary.
\end{definition}

Here, a function $F\colon\prescript{<\omega_1}{}{2}\to A_-$ is ``Borel'' if and only if for every $\delta<\omega_1$ the restriction $F\rest\prescript{\delta}{}{2}$ is Borel.

\section{Almost refinement}\label{section:tre}

This section is dedicated to the study of the almost refinement relation. For a Boolean algebra $\B$, let $\B^+=\B\setminus\{\bbot\}$. Let $\Part(\B)$ denote the set of all \emph{maximal antichains} in $\B$, i.e.\ maximal sets of pairwise disjoint elements of $\B^+$.

\begin{definition} Let $\B$ be a Boolean algebra. Given $A,B\in\Part(\B)$, we say that $B$ \emph{refines} $A$, in symbols $A\le B$, if for all $b\in B$ there exists $a\in A$ such that $b\le a$. The corresponding relational system is defined as
\[
\Partb(\B)=\langle\Part(\B),\le,\Part(\B)\rangle.
\]
\end{definition}

Inspired by Matet's relation \cite[Section 4]{MR0830067} on partitions of $\omega$, we introduce the relation of almost refinement on Boolean algebras satisfying the countable chain condition (c.c.c.). First we need some notation: for a maximal antichain $A$ and a finite subset $F\subseteq A$, let $A^F=(A\setminus F)\cup\{\bigvee F\}$, which is also a maximal antichain.

\begin{definition} Let $\B$ be a c.c.c.\ Boolean algebra. Given $A,B\in\Part(\B)$, we say that $B$ \emph{almost refines} $A$, in symbols $A\le^* B$, if there exists a finite subset $F\subseteq A$ such that $A^F\le B$. The corresponding relational system is defined as
\[
\Partb^*(\B)=\langle\Part(\B),\le^*,\Part(\B)\rangle.
\]
\end{definition}

\begin{remark} $\Part(\B)$ equipped with the relation $\le^*$ defined above is a directed set.
\end{remark}

For each $a\in\B^+$, let $\B\rest a=\Set{b\in\B}{b\le a}$ be the \emph{relative algebra} of $\B$ with respect to $a$.

\begin{lemma}\label{lemma:rel} If $\B$ is a c.c.c.\ Boolean algebra and $a\in\B^+$, then $\Partb^*(\B\rest a)\le_\mathrm{T}\Partb^*(\B)$.
\end{lemma}
\begin{proof} Let $\varphi_-\colon\Part(\B\rest a)\to\Part(\B)$ be a function such that, if $A\in\Part(\B\rest a)$, then $\varphi_-(A)$ is a maximal antichain in $\B$ including $A$. Let $\varphi_+\colon\Part(\B)\to\Part(\B\rest a)$ be defined as follows: if $B\in\Part(\B)$ then $\varphi_+(B)=\Set{a\wedge b}{b\in B}\setminus\{\bbot\}$.

To see that $\varphi_-$ and $\varphi_+$ give a generalized Galois-Tukey connection, suppose $\varphi_-(A)\le^* B$; then there exists a finite subset $F\subseteq\varphi_-(A)$ such that $\varphi_-(A)^F\le B$. It follows that $A^{F\cap A}\le\varphi_+(B)$ and therefore $A\le^*\varphi_+(B)$.
\end{proof}

Next, we would like to determine when the dominating and bounding numbers of $\Partb^*(\B)$ are uncountable. This problem is related to the weak distributivity of~$\B$.

\begin{definition}[von Neumann~\cite{MR0120174}] A Boolean algebra $\B$ is \emph{weakly $\langle\omega,\omega\rangle$-dis\-tributive} if for every function $a\colon\omega\times\omega\to\B$, satisfying $a(n,m)\le a(n,m+1)$ for all $n,m<\omega$, the equality
\[
\bigwedge_{n<\omega}\bigvee_{m<\omega}a(n,m)=\bigvee_{f\in\prescript{\omega}{}{\omega}}\bigwedge_{n<\omega}a(n,f(n))
\]
holds, provided that the relevant suprema and infima exist in $\B$.
\end{definition}

For the purpose of Proposition~\ref{proposition:unc}, we find it more convenient to work with an equivalent formulation of weak distributivity, stated in the following lemma.

\begin{lemma}[Traczyk~\cite{MR0158842}]\label{lemma:traczyk} A Boolean algebra $\B$ is weakly $\langle\omega,\omega\rangle$-distributive if and only if for every set $\Set{A_n}{n<\omega}$ of countable maximal antichains in $\B$ there exists a maximal antichain $B$ such that each element of $B$ meets only finitely many elements of each $A_n$
\end{lemma}

\begin{proposition}\label{proposition:unc} Let $\B$ be a c.c.c.\ Boolean algebra; then
\begin{enumerate}
\item\label{proposition:uncuno} $\dd(\Partb^*(\B))>\aleph_0$ if and only if $\B$ is not atomic;
\item\label{proposition:uncdue} $\bb(\Partb^*(\B))>\aleph_0$ if and only if $\B$ is weakly $\langle\omega,\omega\rangle$-distributive.
\end{enumerate}
\end{proposition}
\begin{proof} \eqref{proposition:uncuno} If $\B$ is atomic, then $\B$ has a maximal antichain consisting of atoms, which immediately implies $\dd(\Partb^*(\B))=1$.

To establish the other implication, suppose first $\B$ is atomless. Let $\Set{B_n}{n<\omega}$ be a set of maximal antichains in $\B$; we shall find a maximal antichain $A$ such that $A\not\le^* B_n$ for all $n<\omega$. Let $\Set{a_n}{n<\omega}$ be a fixed antichain; for each $n<\omega$ let us choose $b\in B_n$ such that $a_n\wedge b>\bbot$. By atomlessness, let $A_n$ be an infinite antichain in $\B\rest(a_n\wedge b)$ and then let $A$ be a maximal antichain in $\B$ including $\bigcup_{n<\omega}A_n$. Towards a contradiction, suppose there exists $n<\omega$ such that $A\le^* B_n$; by definition there exists a finite subset $F\subset A$ such that $A^F\le B_n$. If $a\in A_n\setminus F$, then $a$ is greater than or equal to some element of $B_n$, which is a contradiction. For the general case, if $\B$ is not atomic then there exists $a\in\B^+$ such that $\B\rest a$ is atomless. By what we have shown so far, $\dd(\Partb^*(\B\rest a))>\aleph_0$, but Lemma~\ref{lemma:rel} guarantees that $\dd(\Partb^*(\B\rest a))\le\dd(\Partb^*(\B))$.

\eqref{proposition:uncdue} The left-to-right implication follows from Lemma~\ref{lemma:traczyk} and the observation that, if $B$ almost refines $A$, then each element of $B$ meets only finitely many elements of $A$.

Conversely, suppose $\B$ is weakly $\langle\omega,\omega\rangle$-distributive. Let $\Set{A_n}{n<\omega}$ be a set of maximal antichains in $\B$; we shall find a maximal antichain $B$ such that $A_n\le^* B$ for all $n<\omega$. Without loss of generality, suppose $A_n\le A_{n+1}$ for all $n<\omega$. By Lemma~\ref{lemma:traczyk}, there exists a maximal antichain $\Set{b_i}{i<\omega}$ in $\B$ such that for every $i<\omega$ and every $n<\omega$, the element $b_i$ meets only finitely many elements of $A_n$. Now consider for each $i<\omega$ the finite set
\[
F_i=\Set{a\wedge b_i}{a\in A_i}\setminus\{\bbot\}
\]
and let $B=\bigcup_{i<\omega}F_i$, which is a maximal antichain refining $\Set{b_i}{i<\omega}$. Given $n<\omega$, we have that for all $b\in\bigcup_{i\ge n}F_i$ there exists $a\in A_n$ such that $b\le a$, from which we can conclude that $B$ almost refines $A_n$.
\end{proof}

We further analyse $\Partb^*(\B)$ using generalized Galois-Tukey connections.

\begin{theorem}\label{theorem:density} If $\B$ is a c.c.c.\ Boolean algebra, then
\[
{\langle\B^+,\ge,\B^+\rangle}^\sigma\le_\mathrm{T}\Partb^*(\B)\le_\mathrm{T}\Partb(\B).
\]
\end{theorem}
\begin{proof} To establish ${\langle\B^+,\ge,\B^+\rangle}^\sigma\le_\mathrm{T}\Partb^*(\B)$, we define two functions
\[
\varphi_-\colon\B^+\to\Part(\B)\quad\text{and}\quad\varphi_+\colon\Part(\B)\to\prescript{\omega}{}{\B^+}
\]
such that for every $b\in\B^+$ and every $B\in\Part(\B)$
\begin{equation}\label{eq:dt}
\varphi_-(b)\le^*B\implies(\exists n<\omega)\bigl(\varphi_+(B)(n)\le b\bigr).
\end{equation}

Given $b\in\B^+$, we distinguish two cases: if $\B\rest b$ is atomless, let $\varphi_-(b)$ be any maximal antichain in $\B$ such that $\varphi_-(b)\cap\B\rest b$ is infinite. If not, let $\varphi_-(b)$ be defined arbitrarily. On the other hand, let $A$ be the set of atoms of $\B$. Given $B\in\Part(\B)$, by the c.c.c.\ let $\varphi_+(B)\colon\omega\to A\cup B$ be any surjective function. We check that \eqref{eq:dt} is satisfied: let $b\in\B^+$ and $B\in\Part(\B)$ and suppose that $\varphi_-(b)\le^*B$. If $\B\rest b$ has an atom, then there exists $a\in A$ such that $a\le b$ and we are done. If $\B\rest b$ is atomless, let $F\subset\varphi_-(b)$ be a finite subset such that $\varphi_-(b)^F\le B$. Pick $c\in(\varphi_-(b)\cap\B\rest b)\setminus F$ and note that $c$ must be greater than or equal to some element of $B$, which gives the conclusion.

Finally, to establish $\Partb^*(\B)\le_\mathrm{T}\Partb(\B)$, it is sufficient to observe that the relation of refinement is stronger than the relation of almost refinement, hence the identity gives a generalized Galois-Tukey connection.
\end{proof}

In particular, we obtain that the relational systems $\Partb(\B)$ and $\Partb^*(\B)$, despite not being necessarily Galois-Tukey equivalent, always have the same dominating number.

\begin{corollary}\label{corollary:mis} If $\B$ is a c.c.c.\ Boolean algebra, then $\dd(\Partb^*(\B))=\dd(\Partb(\B))$.
\end{corollary}
\begin{proof} The inequality $\dd(\Partb^*(\B))\le\dd(\Partb(\B))$ follows already from Remark~\ref{remark:vojtas} and Theorem~\ref{theorem:density}.

To prove the converse inequality, if $\B$ is atomic then clearly $\dd(\Partb(\B))=1$ and we are done. Hence, we may assume that $\B$ is not atomic. Let $\mathcal{B}^*\subseteq\Part(\B)$ be a set of cardinality $\dd(\Partb^*(\B))$ such that for all $A\in\Part(\B)$ there exists $B\in\mathcal{B}^*$ such that $A\le^*B$. Similarly, let $\mathcal{D}\subseteq\prescript{\omega}{}{\B^+}$ be a set of cardinality $\dd\bigl({\langle\B^+,\ge,\B^+\rangle}^\sigma\bigr)$ such that for every $b\in\B^+$ there exists $d\in\mathcal{D}$ such that $b\ge^\sigma d$. This condition implies that the set $D=\Set{d(n)}{d\in\mathcal{D}\text{ and }n<\omega}$ is dense in $\B$.

For every $B\in\mathcal{B}^*$ and every finite antichain $F\subset D$, we define a maximal antichain
\[
B_F=F\cup\Set*{b\wedge\neg\bigvee F}{b\in B}\setminus\{\bbot\}
\]
and then we let
\[
\mathcal{B}=\Set{B_F}{B\in\mathcal{B}^*\text{ and }F\text{ is a finite antichain in }D}.
\]
Now, let $A$ be an arbitrary maximal antichain; by density we may assume that $A\subseteq D$. By the choice of $\mathcal{B}^*$, there exists $B\in\mathcal{B}^*$ such that $A\le^*B$, which means there exists a finite subset $F\subseteq A$ such that $A^F\le B$, therefore $A\le B_F$, as desired.

In conclusion, we have shown that
\[
\dd(\Partb(\B))\le\abs{\mathcal{B}}\le\dd(\Partb^*(\B))+\dd\bigl({\langle\B^+,\ge,\B^+\rangle}^\sigma\bigr)+\aleph_0.
\]
However, Proposition~\ref{proposition:unc}.\eqref{proposition:uncuno} gives that $\aleph_0<\dd(\Partb^*(\B))$ and Theorem~\ref{theorem:density} gives that $\dd({\langle\B^+,\ge,\B^+\rangle}^\sigma)\le\dd(\Partb^*(\B))$, whence we can conclude that $\dd(\Partb(\B))\le\dd(\Partb^*(\B))$.
\end{proof}

We move on to discuss the relation with a strong form of the countable chain condition.

\begin{definition}[Horn and Tarski~\cite{MR0028922}] A Boolean algebra $\B$ is \emph{$\sigma$-finite c.c.}\ if there are subsets $S_n\subseteq\B^+$, for $n<\omega$, such that every antichain in $S_n$ is finite and $\B^+=\bigcup_{n<\omega}S_n$.
\end{definition}

As usual, for functions $f,g\in\prescript{\omega}{}{\omega}$, we write $f\le^* g$ to mean than $f(n)\le g(n)$ for all but finitely many $n<\omega$. 

\begin{theorem}\label{theorem:sfcc} If $\B$ is a non-atomic $\sigma$-finite c.c.\ Boolean algebra, then
\[
\langle\prescript{\omega}{}{\omega},\le^*,\prescript{\omega}{}{\omega}\rangle\le_\mathrm{T}\Partb^*(\B).
\]
\end{theorem}
\begin{proof} By Lemma~\ref{lemma:rel}, we may further assume that $\B$ is atomless. We aim to construct two functions
\[
\varphi_-\colon\prescript{\omega}{}{\omega}\to\Part(\B)\quad\text{and}\quad\varphi_+\colon\Part(\B)\to\prescript{\omega}{}{\omega}
\]
such that for every $f\colon\omega\to\omega$ and every $B\in\Part(\B)$ we have
\begin{equation}\label{eq:sftukey}
\varphi_-(f)\le^* B\implies f\le^*\varphi_+(B).
\end{equation}
By hypothesis, there exists a decomposition $\B^+=\bigcup_{n<\omega}S_n$ such that every antichain in $S_n$ is finite. Without loss of generality, for each $n<\omega$ we may further assume that if $b\in S_n$ and $b\le a$ then $a\in S_n$. Finally, we also fix a maximal antichain $\Set{a_i}{i<\omega}$ in $\B$ throughout.

Given $f\colon\omega\to\omega$, we proceed to define $\varphi_-(f)$. First, it is easy to see that for each $i<\omega$ the set $\B^+\setminus\bigcup_{n<f(i)}S_n$ is dense. Indeed, let $a\in\B^+$; by atomlessness, there exists an infinite antichain $B$ such that $a=\bigvee B$. Since all antichains in $S_n$ are finite, there must be some $b\in B\setminus\bigcup_{n<f(i)}S_n$, as desired. Consequently, we can choose for each $i<\omega$ an antichain $A^i_f$ such that $a_i=\bigvee A^i_f$ and $A^i_f\cap\bigcup_{n<f(i)}S_n=\emptyset$; then we define
\[
\varphi_-(f)=\bigcup_{i<\omega}A^i_f.
\]

On the other hand, let $B$ be a maximal antichain in $\B$. We define $\varphi_+(B)\in\prescript{\omega}{}{\omega}$ as follows: given $i<\omega$, let
\[
\varphi_+(B)(i)=\min\Set{n<\omega}{\text{there exists }b\in B\cap S_n\text{ such that }b\le a_i}
\]
if there is such an $n$, otherwise let $\varphi_+(B)(i)=0$.

To establish \eqref{eq:sftukey}, suppose that $\varphi_-(f)\le^* B$. This implies the existence of $i_0<\omega$ such that if $i\ge i_0$ then every element of $A^i_f$ is the supremum of a subset of $B$. Let $i\ge i_0$; as we are in the first case of the definition of $\varphi_+(B)$, there exists $b\in B\cap S_{\varphi_+(B)(i)}$ such that $b\le a_i$. By the assumption on $i$, we can find $a\in A^i_f$ with $b\le a$. By upwards closure we have $a\in S_{\varphi_+(B)(i)}$, but $a\notin\bigcup_{n<f(i)}S_n$, hence $f(i)\le\varphi_+(B)(i)$ as we wanted to show.
\end{proof}

\begin{remark} If $\mathbb{S}$ is a Suslin algebra, then $\dd(\Partb^*(\mathbb{S}))=\aleph_1$. Such an algebra exists in the Cohen model, where $\dd=2^{\aleph_0}$. Thus, the conclusion of Theorem~\ref{theorem:sfcc} does not hold in general for non-atomic c.c.c.\ Boolean algebras.
\end{remark}

Let $\mathcal{B}(\prescript{\omega}{}{2})$ be the $\sigma$-algebra generated by the clopen subsets of the Cantor space $\prescript{\omega}{}{2}$. We consider the nowhere dense ideal $\nwd(\prescript{\omega}{}{2})$ and the meagre ideal $\M$ over $\prescript{\omega}{}{2}$. The \emph{Cohen algebra} is the quotient Boolean algebra $\C_\omega=\mathcal{B}(\prescript{\omega}{}{2})/\M$, thus elements of $\C_\omega$ are represented as equivalence classes $\eq{X}{\M}$ for $X\in\mathcal{B}(\prescript{\omega}{}{2})$. 

For the remainder of this section, we focus on $\Partb^*(\C_\omega)$. First, we fix some standard notation: for every $s\in\prescript{<\omega}{}{2}$, let
\[
N_s=\Set{f\in\prescript{\omega}{}{2}}{s\subset f}.
\]
Furthermore, for every $n<\omega$ we denote
\[
\langle 0^n\rangle=\bigl\langle\underbrace{0,\dots,0}_{n\text{ times}}\bigr\rangle\in\prescript{n}{}{2}\quad\text{and}\quad
\langle 0^n1\rangle=\bigl\langle\underbrace{0,\dots,0}_{n\text{ times}},1\bigr\rangle\in\prescript{n+1}{}{2}.
\]

\begin{lemma}\label{lemma:ref} For every $A\in\Part(\C_\omega)$ there exists $S\subseteq\prescript{<\omega}{}{2}$ such that:
\begin{enumerate}
\item\label{refuno} $\Set{\eq{N_s}{\M}}{s\in S}$ is a maximal antichain refining $A$;
\item\label{refdue} for every $t\in\prescript{<\omega}{}{2}$, if $N_t\subseteq\bigcup_{s\in S}N_s$ then there exists $s\in S$ such that $s\subseteq t$.
\end{enumerate}
\end{lemma}
\begin{proof} Given $A\in\Part(\C_\omega)$, by density there exists $R\subseteq\prescript{<\omega}{}{2}$ such that $\Set{\eq{N_r}{\M}}{r\in R}$ is a maximal antichain refining $A$. Now let
\[
S=\Set{r^\frown\langle 0^n1\rangle}{r\in R\text{ and }n<\omega}.
\]
It is clear that for each $r\in R$
\[
\eq{N_r}{\M}=\bigvee_{n<\omega}\eq*{N_{r^\frown\langle 0^n1\rangle}}{\M}
\]
in $\C_\omega$, hence condition~\eqref{refuno} is satisfied. To prove \eqref{refdue}, suppose $N_t\subseteq\bigcup_{s\in S}N_s$ for some $t\in\prescript{<\omega}{}{2}$. Let $c_0\in\prescript{\omega}{}{2}$ be the sequence with constant value $0$; since $t^\frown c_0\in N_t$, there exists $s\in S$ such that $s\subset t^\frown c_0$. But $s$ is a finite sequence ending in $1$, hence $s\subseteq t$ as desired.
\end{proof}

The next theorem establishes an equivalence between maximal antichains in $\C_\omega$ and nowhere dense subsets of $\prescript{\omega}{}{2}$.

\begin{theorem}\label{theorem:mainc} $\langle\nwd(\prescript{\omega}{}{2}),\subseteq,\nwd(\prescript{\omega}{}{2})\rangle\equiv_\mathrm{T}\Partb^*(\C_\omega)\equiv_\mathrm{T}\Partb(\C_\omega)$.
\end{theorem}
\begin{proof} First, we define a function $\varphi_-\colon\nwd(\prescript{\omega}{}{2})\to\Part(\C_\omega)$. Let $X\in\nwd(\prescript{\omega}{}{2})$; then by definition $D_X=\Set{\eq{N_t}{\M}}{N_t\cap X=\emptyset}$ is a dense subset of $\C_\omega$. So we let $\varphi_-(X)$ be any maximal antichain included in $D_X$.

Next, we define a function $\varphi_+\colon\Part(\C_\omega)\to\nwd(\prescript{\omega}{}{2})$. Let $A\in\Part(\C_\omega)$; by Lemma~\ref{lemma:ref} there exists $S\subseteq\prescript{<\omega}{}{2}$ such that:
\begin{itemize}
\item $\Set*{\eq{N_s}{\M}}{s\in S}$ is a maximal antichain refining $A$;
\item for every $t\in\prescript{<\omega}{}{2}$, if $N_t\subseteq\bigcup_{s\in S}N_s$ then there exists $s\in S$ such that $s\subseteq t$.
\end{itemize}
Then we let
\[
\varphi_+(A)=\prescript{\omega}{}{2}\setminus\bigcup_{s\in S}N_s,
\]
which is clearly nowhere dense.

To establish $\langle\nwd(\prescript{\omega}{}{2}),\subseteq,\nwd(\prescript{\omega}{}{2})\rangle\le_\mathrm{T}\Partb^*(\C_\omega)$, we verify that for all $X\in\nwd(\prescript{\omega}{}{2})$ and all $A\in\Part(\C_\omega)$
\[
\varphi_-(X)\le^* A\implies X\subseteq\varphi_+(A).
\]
By definition, $\varphi_+(A)=\prescript{\omega}{}{2}\setminus\bigcup_{s\in S}N_s$ and $\Set{\eq{N_s}{\M}}{s\in S}$ is a maximal antichain refining $A$. Given $s\in S$, there exists a finite subset $F\subseteq\varphi_-(X)$ such that $\eq{N_s}{\M}\le\bigvee F$. But $\varphi_-(X)\subseteq\Set{\eq{N_t}{\M}}{N_t\cap X=\emptyset}$ and therefore $N_s\cap X=\emptyset$.

Secondly, $\Partb^*(\C_\omega)\le_\mathrm{T}\Partb(\C_\omega)$ follows from Theorem~\ref{theorem:density}.

Finally, to prove that $\Partb(\C_\omega)\le_\mathrm{T}\langle\nwd(\prescript{\omega}{}{2}),\subseteq,\nwd(\prescript{\omega}{}{2})\rangle$, we have to show that
\[
\varphi_+(A)\subseteq X\implies A\le\varphi_-(X).
\]
Let $t\in\prescript{<\omega}{}{2}$ be such that $\eq{N_t}{\M}\in\varphi_-(X)$; this implies that
\[
N_t\subseteq\prescript{\omega}{}{2}\setminus X\subseteq\bigcup_{s\in S}N_s.
\]
By the property of $S$, there exists $s\in S$ such that $s\subseteq t$, hence $\eq{N_t}{\M}\le\eq{N_s}{\M}\le a$ for some $a\in A$.
\end{proof}

\begin{corollary}\label{corollary:cohen} $\cof(\M)=\dd(\Partb^*(\C_\omega))$ and $\add(\M)=\bb\bigl({\Partb^*(\C_\omega)}^\sigma\bigr)$.
\end{corollary}
\begin{proof} This follows from Proposition~\ref{proposition:sigma}, Theorem~\ref{theorem:mainc}, and the theorem, due to Fremlin~\cite[Theorem 3B.(b)]{MR1258546}, that ${\langle\nwd(\prescript{\omega}{}{2}),\subseteq,\nwd(\prescript{\omega}{}{2})\rangle}^\sigma\equiv_\mathrm{T}\langle\M,\subseteq,\M\rangle$.
\end{proof}

It would be interesting to determine $\Partb^*(\B)$ for other c.c.c.\ forcings on the reals. For the random algebra $\B_\omega=\mathcal{B}(\prescript{\omega}{}{2})/\N$, we remark that $\langle\N,\subseteq,\N\rangle\le_\mathrm{T}\Partb^*(\B_\omega)$; this follows from Theorem~\ref{theorem:density} and the fact, again due to Fremlin~\cite{MR1258546}, that ${\langle\B_\omega^+,\ge,\B_\omega^+\rangle}^\sigma\equiv_\mathrm{T}\langle\N,\subseteq,\N\rangle$. In analogy with Corollary~\ref{corollary:cohen}, we may ask:

\begin{question} Is $\cof(\N)=\dd(\Partb^*(\B_\omega))$?
\end{question}

\section{Reaping}\label{section:quattro}

In this section, we investigate the reaping relation, focusing in particular on reduced powers of Boolean algebras.

\begin{definition} Given a Boolean algebra $\B$, the \emph{reaping relation} of $\B$ is defined as
\[
\mathbf{R}(\B)=\langle\B,R,\B^+\rangle,
\]
where $b\mathrel{R}r$ if and only if either $r\le b$ or $r\wedge b=\bbot$.

Let $\rr(\B)=\dd(\mathbf{R}(\B))$ be the \emph{reaping number} of $\B$, also known in the literature~\cite{MR0961402} as the \emph{weak density} of $\B$. In case the algebra $\B$ is atomless, we also define $\sss(\B)=\bb(\mathbf{R}(\B))$ to be the \emph{splitting number} of $\B$. Following common usage, instead of $\rr(\mathcal{P}(\omega)/\mathrm{fin})$ and $\sss(\mathcal{P}(\omega)/\mathrm{fin})$, we write simply $\rr$ and $\sss$ to denote the classic reaping and splitting numbers.
\end{definition}

A result of Monk shows that the reaping and splitting numbers of the Cohen algebra are both countable.

\begin{proposition}[{Monk~\cite[Proposition 26]{MR1877033}}] $\rr(\C_\omega)=\sss(\C_\omega)=\aleph_0$.
\end{proposition}

Reduced powers of Boolean algebras have been studied, e.g.\ by Brendle~\cite{MR2412898} and more recently by Kurili{\'c}~\cite{MR4874856}. We recall some relevant terminology. Given a Boolean algebra $\B$, the infinite product $\prescript{\omega}{}{\B}$ is also a Boolean algebra with the pointwise operations. Let
\[
\mathrm{Fin}=\Set{f\in\prescript{\omega}{}{\B}}{\text{the set }\Set{n<\omega}{f(n)>\bbot}\text{ is finite}},
\]
which is easily seen to be an ideal on $\prescript{\omega}{}{\B}$. The \emph{reduced power} of $\B$ is the quotient algebra $\prescript{\omega}{}{\B}/\mathrm{Fin}$. 

Balcar and Hru{\v s}{\'a}k~\cite[Proposition 2.1]{MR2130854} showed that $\mathcal{P}(\omega)/\mathrm{fin}$ can be completely embedded into $\prescript{\omega}{}{\B}/\mathrm{Fin}$ for every Boolean algebra $\B$. In the next lemma, we elaborate on this idea to derive a consequence for the reaping relation.

\begin{proposition}\label{proposition:reap} If $\B$ is a Boolean algebra, then $\mathbf{R}(\mathcal{P}(\omega)/\mathrm{fin})\le_\mathrm{T}\mathbf{R}(\prescript{\omega}{}{\B}/\mathrm{Fin})$.
\end{proposition}
\begin{proof} We define a homomorphism $\varphi_-\colon\mathcal{P}(\omega)/\mathrm{fin}\to\prescript{\omega}{}{\B}/\mathrm{Fin}$ as follows: for each $X\subseteq\omega$, let $\chi_X\colon\omega\to\B$ be the characteristic function of $X$, that is,
\[
\chi_X(n)=
\begin{cases}
\btop & \text{if }n\in X \\
\bbot & \text{if }n\notin X
\end{cases}\, .
\]
Then we let $\varphi_-(\eq{X}{\mathrm{fin}})=\eq{\chi_X}{\mathrm{Fin}}$. On the other hand, let $\varphi_+\colon{(\prescript{\omega}{}{\B}/\mathrm{Fin})}^+\to{(\mathcal{P}(\omega)/\mathrm{fin})}^+$ be defined as follows: if $f\in\prescript{\omega}{}{\B}\setminus\mathrm{Fin}$, then $\varphi_+(\eq{f}{\mathrm{Fin}})=\eq{\Set{n<\omega}{f(n)>\bbot}}{\mathrm{fin}}$.

For all $X\subseteq\omega$ and $f\in\prescript{\omega}{}{\B}\setminus\mathrm{Fin}$, we have
\[
\chi_X\wedge f\in\mathrm{Fin}\implies\Set{n\in X}{f(n)>\bbot}\in\mathrm{fin},
\]
which easily implies that $\varphi_-$ and $\varphi_+$ give a generalized Galois-Tukey connection from $\mathbf{R}(\mathcal{P}(\omega)/\mathrm{fin})$ to $\mathbf{R}(\prescript{\omega}{}{\B}/\mathrm{Fin})$.
\end{proof}

Next, we establish a general upper bound which involves the almost refinement relation discussed in the previous section.

\begin{theorem}\label{theorem:bound} If $\B$ is a c.c.c.\ Boolean algebra, then
\[
\mathbf{R}(\prescript{\omega}{}{\B}/\mathrm{Fin})\le_\mathrm{T}\Partb^*(\B)\mathbin{;}\mathbf{R}(\mathcal{P}(\omega)/\mathrm{fin}).
\]
\end{theorem}
\begin{proof} If $\B$ is finite, then $\B=\mathcal{P}(n)$ for some $n<\omega$ and therefore $\prescript{\omega}{}{\B}/\mathrm{Fin}$ is isomorphic to $\mathcal{P}(\omega)/\mathrm{fin}$. It follows that $\mathbf{R}(\prescript{\omega}{}{\B}/\mathrm{Fin})\equiv_\mathrm{T}\mathbf{R}(\mathcal{P}(\omega)/\mathrm{fin})$, from which the conclusion is immediate. If $\B$ is infinite, let us fix a maximal antichain $\Set{a_n}{n<\omega}$ in $\B$. Furthermore, for every $B\in\Part(\B)$ and $n<\omega$, choose $B(n)\in B$ such that $a_n\wedge B(n)>\bbot$.

Given $g\colon\omega\to\B$, we define a maximal antichain
\[
A_g=\bigcup_{n<\omega}\{a_n\wedge g(n),a_n\wedge\neg g(n)\}\setminus\{\bbot\}
\]
and a function
\[
\begin{split}
\beta_g\colon\Part(\B)&\longrightarrow\mathcal{P}(\omega)/\mathrm{fin}\\
B &\longmapsto\eq{\Set{n<\omega}{B(n)\le g(n)}}{\mathrm{fin}}
\end{split}\, ;
\]
then, we let
\[
\begin{split}
\varphi_-\colon\prescript{\omega}{}{\B}/\mathrm{Fin}&\longrightarrow\Part(\B)\times\prescript{\Part(\B)}{}{\mathcal{P}(\omega)/\mathrm{fin}}\\
\eq{g}{\mathrm{Fin}} &\longmapsto\langle A_g,\beta_g\bigr\rangle
\end{split}\, .
\]

On the other hand, given $B\in\Part(\B)$ and an infinite $X\subseteq\omega$, we define a function $f_{B,X}\colon\omega\to\B$ such that for each $n<\omega$
\[
f_{B,X}(n)=
\begin{cases}
B(n) & \text{if }n\in X \\
\bbot & \text{if }n\notin X
\end{cases}\, ;
\]
then, we let
\[
\begin{split}
\varphi_+\colon\Part(\B)\times{(\mathcal{P}(\omega)/\mathrm{fin})}^+&\longrightarrow{(\prescript{\omega}{}{\B}/\mathrm{Fin})}^+\\
\langle B,\eq{X}{\mathrm{fin}}\rangle &\longmapsto\eq{f_{B,X}}{\mathrm{Fin}}
\end{split}\, .
\]

To show that $\varphi_-$ and $\varphi_+$ form indeed a generalized Galois-Tukey connection, we prove that for all $g\colon\omega\to\B$, $B\in\Part(\B)$, and infinite $X\subseteq\omega$
\[
A_g\le^*B\text{ and }\beta_g(B)\mathrel{R}\eq{X}{\mathrm{fin}}\implies\eq{g}{\mathrm{Fin}}\mathrel{R}\eq{f_{B,X}}{\mathrm{Fin}}.
\]
Since $A_g\le^*B$, there exists a finite subset $F\subset A_g$ such that $A_g^F\le B$. Let $k<\omega$ be sufficiently large that for every $a\in F$ there exists $n<k$ such that $a\le a_n$. Then it is clear that
\begin{equation}\label{eq:gtbb}
(\forall n\ge k)(B(n)\le g(n)\text{ or }B(n)\wedge g(n)=\bbot).
\end{equation}
Since $\beta_g(B)\mathrel{R}\eq{X}{\mathrm{fin}}$, we have two possibilities: if $\eq{X}{\mathrm{fin}}\le\beta_g(B)$, then $\eq{f_{B,X}}{\mathrm{Fin}}\le\eq{g}{\mathrm{Fin}}$ and we are already done. If not, then the set $\Set{n\in X}{B(n)\le g(n)}$ is finite which, combined with \eqref{eq:gtbb}, yields that the set $\Set{n\in X}{B(n)\wedge g(n)>\bbot}$ is also finite. Therefore $\eq{f_{B,X}}{\mathrm{Fin}}\wedge\eq{g}{\mathrm{Fin}}=\bbot$, as desired.
\end{proof}

For the remainder of this section, we shall be concerned with the topological space of the rationals which, for convenience, we identify with the binary tree $\prescript{<\omega}{}{2}$. More precisely, it is possible to define a dense linear ordering $<$ on $\prescript{<\omega}{}{2}$ as follows: $s<t$ if and only if $t\subset s$ and $s(\abs{t})=0$; or $s\subset t$ and $t(\abs{s})=1$; or $s$ and $t$ are incompatible and $s(k)<t(k)$, where $k$ is their first point of difference. In the space $\langle\prescript{<\omega}{}{2},<\rangle$ equipped with the order topology, we have that:
\begin{itemize}
\item $D\subseteq\prescript{<\omega}{}{2}$ is dense if and only if for all $s\in\prescript{<\omega}{}{2}$ there exists $t\in D$ such that $s\subseteq t$;
\item $N\subseteq\prescript{<\omega}{}{2}$ is nowhere dense if and only if there exists a maximal antichain $A\subseteq\prescript{<\omega}{}{2}$ such that for all $s\in N$ there exists $t\in A$ such that $s\subseteq t$.
\end{itemize}
For further details, we refer the reader to \cite[Fact 1.1]{MR2098150}.

\begin{definition}\label{definition:rel} Let $\Dense(\prescript{<\omega}{}{2})$ be the set of dense subsets of $\prescript{<\omega}{}{2}$ and let $\nwd(\prescript{<\omega}{}{2})$ be the ideal of nowhere dense subsets of $\prescript{<\omega}{}{2}$. We define the relational system
\[
\mathbf{D}(\prescript{<\omega}{}{2})=\langle\nwd(\prescript{<\omega}{}{2}),\perp^*,\Dense(\prescript{<\omega}{}{2})\rangle,
\]
where $N\perp^*D$ if and only if $N\cap D$ is finite.
\end{definition}

The bounding number of $\mathbf{D}(\prescript{<\omega}{}{2})$ has been computed by Keremedis~\cite{MR1234629}; nine years later, Balcar, Hern{\'a}ndez-Hern{\'a}ndez, and Hru{\v s}{\'a}k~\cite{MR2098150} computed the dominating number of $\mathbf{D}(\prescript{<\omega}{}{2})$. We summarize both in the next theorem.

\begin{theorem}[{Keremedis~\cite[Theorem 2]{MR1234629}; Balcar, Hern{\'a}ndez-Hern{\'a}ndez, and Hru{\v s}{\'a}k~\cite[Theorem 1.6]{MR2098150}}]\label{theorem:ker} $\cof(\M)=\dd(\mathbf{D}(\prescript{<\omega}{}{2}))$ and $\add(\M)=\bb(\mathbf{D}(\prescript{<\omega}{}{2}))$.
\end{theorem}

Instead of dense sets of rationals, we may equivalently use divergent sequences. More precisely, let
\[
\suc(\prescript{<\omega}{}{2})=\Set*{\Seq{x_i}{i<\omega}\in\prescript{\omega}{}{(\prescript{<\omega}{}{2})}}{(\forall s\in\prescript{<\omega}{}{2})(\exists k<\omega)(\forall i\ge k)(x_i<s)}
\]
and, in analogy with Definition~\ref{definition:rel}, consider the relational system $\langle\nwd(\prescript{<\omega}{}{2}),\perp^*,\allowbreak\suc(\prescript{<\omega}{}{2})\rangle$, where $N\perp^*\Seq{x_i}{i<\omega}$ if and only if $\Set{i<\omega}{x_i\in N}$ is finite. Then, we have the following Galois-Tukey equivalence.

\begin{lemma}\label{lemma:seq} $\langle\nwd(\prescript{<\omega}{}{2}),\perp^*,\suc(\prescript{<\omega}{}{2})\rangle\equiv_\mathrm{T}\mathbf{D}(\prescript{<\omega}{}{2})$.
\end{lemma}
\begin{proof} First, it is easy to see that $\langle\nwd(\prescript{<\omega}{}{2}),\perp^*,\suc(\prescript{<\omega}{}{2})\rangle\le_\mathrm{T}\mathbf{D}(\prescript{<\omega}{}{2})$. Indeed, we can construct a pair of functions
\[
\varphi_-\colon\nwd(\prescript{<\omega}{}{2})\to\nwd(\prescript{<\omega}{}{2})\quad\text{and}\quad\varphi_+\colon\Dense(\prescript{<\omega}{}{2})\to\suc(\prescript{<\omega}{}{2})
\]
as follows: let $\varphi_-$ be the identity function. Given $D\in\Dense(\prescript{<\omega}{}{2})$, for each $i<\omega$ choose $x_i\in D$ such that $\langle 0^i\rangle\subseteq x_i$ and define $\varphi_+(D)=\Seq{x_i}{i<\omega}$, which belongs to $\suc(\prescript{<\omega}{}{2})$ by construction. If $N\cap D$ is finite, then $\Set{i<\omega}{x_i\in N}$ is finite, as we wanted.

Secondly, we prove that $\mathbf{D}(\prescript{<\omega}{}{2})\le_\mathrm{T}\langle\nwd(\prescript{<\omega}{}{2}),\perp^*,\suc(\prescript{<\omega}{}{2})\rangle$. Let $\prescript{<\omega}{}{2}=\Set{s_n}{n<\omega}$ be a fixed enumeration of the rationals and define a function
\[
\begin{split}
\psi_-\colon\nwd(\prescript{<\omega}{}{2})&\longrightarrow\nwd(\prescript{<\omega}{}{2})\\
N &\longmapsto\bigcup_{n<\omega}\Set*{s\in\prescript{<\omega}{}{2}}{\langle 0^n\rangle\subseteq s\text{ and }s_n^\frown s\in N}
\end{split}\, .
\]
To see that $\psi_-(N)$ is indeed nowhere dense, first note that for each $k<\omega$ the set
\[
\psi_-(N)\cap\Set*{s\in\prescript{<\omega}{}{2}}{\langle 0^k1\rangle\subseteq s}=\bigcup_{n\le k}\Set*{s\in\prescript{<\omega}{}{2}}{\langle 0^k1\rangle\subseteq s\text{ and }s_n^\frown s\in N}
\]
is nowhere dense, being a finite union of nowhere dense sets. It follows that $\psi_-(N)\cap\bigcup_{k<\omega}\Set*{s\in\prescript{<\omega}{}{2}}{\langle 0^k1\rangle\subseteq s}$ is also nowhere dense, being a union of nowhere dense sets separated by disjoint open sets. But since $\bigcup_{k<\omega}\Set*{s\in\prescript{<\omega}{}{2}}{\langle 0^k1\rangle\subseteq s}$ is open dense in $\prescript{<\omega}{}{2}$, we conclude that $\psi_-(N)$ itself is nowhere dense.

On the other hand, we define
\[
\begin{split}
\psi_+\colon\suc(\prescript{<\omega}{}{2})&\longrightarrow\Dense(\prescript{<\omega}{}{2})\\
\Seq{x_i}{i<\omega}&\longmapsto\bigcup_{n<\omega}\Set{s_n^\frown x_i}{i<\omega\text{ and }\langle 0^n\rangle\subseteq x_i}
\end{split}\, .
\]
We observe that $\psi_+(\Seq{x_i}{i<\omega})$ is dense because, for all $n<\omega$, there exists $i<\omega$ such that $\langle 0^n\rangle\subseteq x_i$; hence $s_n^\frown x_i\in\psi_+(\Seq{x_i}{i<\omega})$ and obviously $s_n\subseteq s_n^\frown x_i$.

The proof is complete once we show that, for all $N\in\nwd(\prescript{<\omega}{}{2})$ and $\Seq{x_i}{i<\omega}\in\suc(\prescript{<\omega}{}{2})$,
\begin{equation}\label{eq:lhs}
\Set{i<\omega}{x_i\in\psi_-(N)}\text{ is finite}\implies N\cap\psi_+(\Seq{x_i}{i<\omega})\text{ is finite}.
\end{equation}
If the set on the left-hand side of \eqref{eq:lhs} is finite, then in particular for each $n<\omega$ the set $N\cap\Set{s_n^\frown x_i}{i<\omega\text{ and }\langle 0^n\rangle\subseteq x_i}$ must be finite, for otherwise there would exist infinitely many $i<\omega$ such that $\langle 0^n\rangle\subseteq x_i$ and $s_n^\frown x_i\in N$, which results in $x_i\in\psi_-(N)$, a contradiction. Now, let $k<\omega$ be sufficiently large that for all $i<\omega$, if $x_i\in\psi_-(N)$ then $\abs{x_i}\le k$. It follows that
\[
N\cap\psi_+(\Seq{x_i}{i<\omega})=\bigcup_{n\le k}N\cap\Set{s_n^\frown x_i}{i<\omega\text{ and }\langle 0^n\rangle\subseteq x_i},
\]
hence the set on the right-hand side of \eqref{eq:lhs} is finite, being a finite union of finite sets.
\end{proof}

The next theorem, together with Proposition~\ref{proposition:reap}, essentially determines the reaping relation of the reduced power of $\C_\omega$.

\begin{theorem}\label{theorem:reap} $\mathbf{D}(\prescript{<\omega}{}{2})\le_\mathrm{T}\mathbf{R}(\prescript{\omega}{}{\C_\omega}/\mathrm{Fin})\le_\mathrm{T}\mathbf{D}(\prescript{<\omega}{}{2})\mathbin{;}\mathbf{R}(\mathcal{P}(\omega)/\mathrm{fin})$.
\end{theorem}
\begin{proof} First, we prove that $\mathbf{D}(\prescript{<\omega}{}{2})\le_\mathrm{T}\mathbf{R}(\prescript{\omega}{}{\C_\omega}/\mathrm{Fin})$. By Lemma~\ref{lemma:seq}, we may equivalently prove that $\langle\nwd(\prescript{<\omega}{}{2}),\perp^*,\suc(\prescript{<\omega}{}{2})\rangle\le_\mathrm{T}\mathbf{R}(\prescript{\omega}{}{\C_\omega}/\mathrm{Fin})$. For each $N\in\nwd(\prescript{<\omega}{}{2})$, choose a maximal antichain $A_N\subseteq\prescript{<\omega}{}{2}$ such that:
\begin{itemize}
\item for all $s\in N$ there exists $t\in A_N$ such that $s\subseteq t$;
\item for all $t\in A_N$ there exists $n<\omega$ such that $\langle 0^n1\rangle\subseteq t$.
\end{itemize}
For every $N\in\nwd(\prescript{<\omega}{}{2})$, we define a function
\[
\begin{split}
g_N\colon\omega&\longrightarrow\C_\omega\\
n&\longmapsto\bigvee\Set*{\eq*{N_{r^\frown\langle 0\rangle}}{\M}}{r\in\prescript{<\omega}{}{2}\text{ and }\langle 0^n1\rangle^\frown r\in A_N}
\end{split}
\]
and finally let
\[
\begin{split}
\varphi_-\colon\nwd(\prescript{<\omega}{}{2})&\longrightarrow\prescript{\omega}{}{\C_\omega}/\mathrm{Fin}\\
N &\longmapsto\eq{g_N}{\mathrm{Fin}}
\end{split}\, .
\]

On the other hand, given $f\in\prescript{\omega}{}{\C_\omega}\setminus\mathrm{Fin}$, pick an increasing sequence $\Seq{k_i}{i<\omega}$ of natural numbers such that $f(k_i)>\bbot$ for all $i<\omega$. By density, for each $i<\omega$ we can choose $x_i\in\prescript{<\omega}{}{2}$ such that $\eq{N_{x_i}}{\M}\le f(k_i)$. Then, we define
\[
\begin{split}
\varphi_+\colon{(\prescript{\omega}{}{\C_\omega}/\mathrm{Fin})}^+&\longrightarrow\suc(\prescript{<\omega}{}{2})\\
\eq{f}{\mathrm{Fin}}&\longmapsto\Seq*{\langle 0^{k_i}1\rangle^\frown x_i}{i<\omega}
\end{split}\, .
\]

The key point is that, for all $N\in\nwd(\prescript{<\omega}{}{2})$, $f\in\prescript{\omega}{}{\C_\omega}\setminus\mathrm{Fin}$, and $i<\omega$
\begin{equation}\label{eq:key}
f(k_i)\le g_N(k_i)\text{ or }f(k_i)\wedge g_N(k_i)=\bbot\implies\langle 0^{k_i}1\rangle^\frown x_i\notin N.
\end{equation}
Indeed, if $\langle 0^{k_i}1\rangle^\frown x_i\in N$, then there exists $t\in A_N$ such that $\langle 0^{k_i}1\rangle^\frown x_i\subseteq t$. Let $r\in\prescript{<\omega}{}{2}$ be such that $t=\langle 0^{k_i}1\rangle^\frown r$; then easily $\bbot<\eq*{N_{r^\frown\langle 0\rangle}}{\M}\le f(k_i)\wedge g_N(k_i)$ and $\bbot<\eq*{N_{r^\frown\langle 1\rangle}}{\M}\le f(k_i)\wedge\neg g_N(k_i)$. Using \eqref{eq:key}, it is immediate to deduce that for all $N\in\nwd(\prescript{<\omega}{}{2})$ and $f\in\prescript{\omega}{}{\C_\omega}\setminus\mathrm{Fin}$
\[
\eq{g_N}{\mathrm{Fin}}\mathrel{R}\eq{f}{\mathrm{Fin}}\implies\Set*{i<\omega}{\langle 0^{k_i}1\rangle^\frown x_i\in N}\text{ is finite},
\]
which means that $\varphi_-$ and $\varphi_+$ give the desired generalized Galois-Tukey connection.

Secondly, to show that $\mathbf{R}(\prescript{\omega}{}{\C_\omega}/\mathrm{Fin})\le_\mathrm{T}\mathbf{D}(\prescript{<\omega}{}{2})\mathbin{;}\mathbf{R}(\mathcal{P}(\omega)/\mathrm{fin})$, let $\Set{a_n}{n<\omega}$ be a fixed maximal antichain in $\C_\omega$. Moreover, for every $D\in\Dense(\prescript{<\omega}{}{2})$ and $n<\omega$, choose $D(n)\in D$ such that $\eq*{N_{D(n)}}{\M}\le a_n$.

Given $g\colon\omega\to\C_\omega$, we define a nowhere dense set
\[
N_g=\Set*{s\in\prescript{<\omega}{}{2}}{\lnot(\exists n<\omega)\bigl(\eq{N_s}{\M}\le a_n\wedge g(n)\text{ or }\eq{N_s}{\M}\le a_n\wedge \neg g(n)\bigr)}
\]
and a function
\[
\begin{split}
\gamma_g\colon\Dense(\prescript{<\omega}{}{2})&\longrightarrow\mathcal{P}(\omega)/\mathrm{fin}\\
D &\longmapsto\eq{\Set{n<\omega}{D(n)\le g(n)}}{\mathrm{fin}}
\end{split}\, ;
\]
then, we define
\[
\begin{split}
\psi_-\colon\prescript{\omega}{}{\C_\omega}/\mathrm{Fin}&\longrightarrow\nwd(\prescript{<\omega}{}{2})\times\prescript{\Dense(\prescript{<\omega}{}{2})}{}{\mathcal{P}(\omega)/\mathrm{fin}}\\
\eq{g}{\mathrm{Fin}} &\longmapsto\langle N_g,\gamma_g\bigr\rangle
\end{split}\, .
\]

On the other hand, given $D\in\Dense(\prescript{<\omega}{}{2})$ and an infinite $X\subseteq\omega$, we define a function $f_{D,X}\colon\omega\to\C_\omega$ as follows: for all $n<\omega$
\[
f_{D,X}(n)=
\begin{cases}
\eq*{N_{D(n)}}{\M} & \text{if }n\in X \\
\bbot & \text{if }n\notin X
\end{cases}\, ;
\]
then, we let
\[
\begin{split}
\psi_+\colon\Dense(\prescript{<\omega}{}{2})\times{(\mathcal{P}(\omega)/\mathrm{fin})}^+&\longrightarrow{(\prescript{\omega}{}{\C_\omega}/\mathrm{Fin})}^+\\
\langle D,\eq{X}{\mathrm{fin}}\rangle &\longmapsto\eq{f_{D,X}}{\mathrm{Fin}}
\end{split}\, .
\]

To show that $\psi_-$ and $\psi_+$ form indeed a generalized Galois-Tukey connection, we have to show that for all $g\colon\omega\to\C_\omega$, $D\in\Dense(\prescript{<\omega}{}{2})$, and infinite $X\subseteq\omega$
\begin{equation}\label{eq:gtbc}
N_g\cap D\text{ is finite and }\gamma_g(D)\mathrel{R}\eq{X}{\mathrm{fin}}\implies\eq{g}{\mathrm{Fin}}\mathrel{R}\eq{f_{D,X}}{\mathrm{Fin}}.
\end{equation}
However, if $N_g\cap D$ is finite, then there exists $k<\omega$ such that
\[
(\forall n\ge k)\bigl(\eq*{N_{D(n)}}{\M}\le g(n)\text{ or }\eq*{N_{D(n)}}{\M}\wedge g(n)=\bbot\bigr),
\]
from which \eqref{eq:gtbc} follows as in the conclusion of Theorem~\ref{theorem:bound}.
\end{proof}

\begin{remark} The structure of the second part of the proof of Theorem~\ref{theorem:reap} is similar to that of Theorem~\ref{theorem:bound}, however using the density gives a sharper bound. In fact, Theorem~\ref{theorem:bound} is enough for the inequality $\rr(\prescript{\omega}{}{\C_\omega}/\mathrm{Fin})\le\rr+\cof(\M)$, but not enough for $\sss(\prescript{\omega}{}{\C_\omega}/\mathrm{Fin})\ge\min\{\sss,\add(\M)\}$.
\end{remark}

\begin{corollary}\label{corollary:reap} $\rr(\prescript{\omega}{}{\C_\omega}/\mathrm{Fin})=\rr+\cof(\M)$ and $\sss(\prescript{\omega}{}{\C_\omega}/\mathrm{Fin})=\min\{\sss,\add(\M)\}$.
\end{corollary}
\begin{proof} The two equalities follow from Proposition~\ref{proposition:reap}, Theorem~\ref{theorem:ker}, and Theorem~\ref{theorem:reap}.
\end{proof}

\section{Ultrafilter numbers}\label{section:cinque}

This section contains some applications to the ultrafilter number of Boolean algebras, which follow from the results of the previous section, in particular Corollary~\ref{corollary:reap}, as well as from the parametrized diamond principle of Definition~\ref{def:diam}.

The following definition is standard: see, for instance, Monk~\cite[Definition (H)]{MR1877033}.

\begin{definition} If $\B$ is an infinite Boolean algebra, let
\[
\uu(\B)=\min\Set{\cof(\langle U,\ge\rangle)}{U\text{ is a non-principal ultrafilter on }\B}
\]
be the \emph{ultrafilter number} of $\B$. Following the usual notation, let $\uu=\uu(\mathcal{P}(\omega)/\mathrm{fin})$.
\end{definition}

\begin{remark}\label{remark:rruu} It is immediate from the definitions that $\rr(\B)\le\uu(\B)$ whenever $\B$ is infinite.
\end{remark}

The next proposition implies that, in many cases, the ultrafilter number is preserved by reduced powers.

\begin{proposition}\label{proposition:ult} If $\B$ is a Boolean algebra, then $\uu\le\uu(\prescript{\omega}{}{\B}/\mathrm{Fin})$. If, in addition, $\B$ is complete, atomless, and c.c.c., then $\uu(\prescript{\omega}{}{\B}/\mathrm{Fin})=\uu(\B)$.
\end{proposition}
\begin{proof} First, to show that $\uu\le\uu(\prescript{\omega}{}{\B}/\mathrm{Fin})$, consider the functions $\varphi_-\colon\mathcal{P}(\omega)/\mathrm{fin}\to\prescript{\omega}{}{\B}/\mathrm{Fin}$ and $\varphi_+\colon{(\prescript{\omega}{}{\B}/\mathrm{Fin})}^+\to{(\mathcal{P}(\omega)/\mathrm{fin})}^+$ defined in the proof of Proposition~\ref{proposition:reap}. Given an ultrafilter $U$ on $\prescript{\omega}{}{\B}/\mathrm{Fin}$, it is easy to check that $\varphi_-^{-1}[U]$ is a non-principal ultrafilter on $\mathcal{P}(\omega)/\mathrm{fin}$ and that, if $C$ is cofinal in $\langle U,\ge\rangle$, then $\varphi_+[C]$ is cofinal in $\bigl\langle\varphi_-^{-1}[U],\ge\bigr\rangle$.

We assume henceforth that $\B$ is a complete atomless c.c.c.\ Boolean algebra and deduce that $\uu(\prescript{\omega}{}{\B}/\mathrm{Fin})=\uu(\B)$. For the inequality $\uu(\B)\le\uu(\prescript{\omega}{}{\B}/\mathrm{Fin})$, consider the homomorphism of Boolean algebras
\[
\begin{split}
\hat{e}\colon\B&\longrightarrow\prescript{\omega}{}{\B}/\mathrm{Fin}\\
b&\longmapsto\eq{\Seq{b}{n<\omega}}{\mathrm{Fin}}
\end{split}\, .
\]
If $U$ is any ultrafilter on $\prescript{\omega}{}{\B}/\mathrm{Fin}$, then $\hat{e}^{-1}[U]$ is an ultrafilter on $\B$, since $\hat{e}$ is a homomorphism. Moreover, $\hat{e}^{-1}[U]$ is not principal, since $\B$ is atomless. Finally, if $C$ is cofinal in $\langle U,\ge\rangle$, then the set
\[
\Set*{\bigvee\Set{f(n)}{n\ge k}}{\eq{f}{\mathrm{Fin}}\in C\text{ and }k<\omega}
\]
is cofinal in $\bigl\langle\hat{e}^{-1}[U],\ge\bigr\rangle$. This argument establishes that $\uu(\B)\le\uu(\prescript{\omega}{}{\B}/\mathrm{Fin})+\aleph_0$ but, since the latter is clearly infinite, we can conclude that $\uu(\B)\le\uu(\prescript{\omega}{}{\B}/\mathrm{Fin})$.

Lastly, we prove that $\uu(\prescript{\omega}{}{\B}/\mathrm{Fin})\le\uu(\B)$. If $V$ is a non-principal ultrafilter on $\B$, then there exists a maximal antichain in $\B$ disjoint from $V$, which we can enumerate as $\Set{a_n}{n<\omega}$. It follows that
\[
U=\Set*{\eq{f}{\mathrm{Fin}}\in\prescript{\omega}{}{\B}/\mathrm{Fin}}{\bigvee\Set{f(n)\wedge a_n}{n<\omega}\in V}
\]
is a non-principal ultrafilter on $\prescript{\omega}{}{\B}/\mathrm{Fin}$. Furthermore, if $D$ is cofinal in $\langle V,\ge\rangle$, then clearly
\[
\Set{\eq{\Seq{d\wedge a_n}{n<\omega}}{\mathrm{Fin}}}{d\in D}
\]
is cofinal in $\langle U,\ge\rangle$.
\end{proof}

Recall that a Boolean algebra $\B$ is \emph{Borel} if the domain of $\B$ is a Borel subset of $\prescript{\omega}{}{2}$ and, moreover, the order relation $\le$ and the incompatibility relation $\perp$ are Borel subsets of $\prescript{\omega}{}{2}\times\prescript{\omega}{}{2}$. For the purpose of the next theorem, however, we shall need a stronger notion which we call ``Borel-homogeneity''.

\begin{definition} A Boolean algebra $\B$ is \emph{Borel-homogeneous} if:
\begin{itemize}
\item $\B$ is a Borel Boolean algebra,
\item the meet operation $\wedge\colon\B\times\B\to\B$ is a Borel function,
\item for each $b\in\B^+$ the relative algebra $\B\rest b$ is isomorphic to $\B$ via a Borel function.
\end{itemize}
\end{definition}

Moore, Hru{\v s}{\'a}k, and D{\v z}amonja~\cite[Theorem 7.8]{MR2048518} showed, in particular, that $\diamondsuit(\mathbf{R}(\mathcal{P}(\omega)/\mathrm{fin}))$ implies $\uu=\aleph_1$. We generalize their argument from $\mathcal{P}(\omega)/\mathrm{fin}$ to reduced powers of Borel-homogeneous Boolean algebras.

\begin{theorem}\label{theorem:diamond} If $\B$ is a Borel Boolean algebra, then $\mathbf{R}(\prescript{\omega}{}{\B}/\mathrm{Fin})$ is a Borel relational system. If, in addition, $\B$ is Borel-homogeneous, then $\diamondsuit(\mathbf{R}(\prescript{\omega}{}{\B}/\mathrm{Fin}))$ implies $\uu(\prescript{\omega}{}{\B}/\mathrm{Fin})=\aleph_1$.
\end{theorem}
\begin{proof} The first assertion is straightforward from the definitions. Suppose in addition that $\B$ is Borel-homogeneous and fix, for each $b\in\B^+$, a Borel isomorphism $\varphi_b\colon\B\to\B\rest b$. For convenience, for each $\omega\le\delta<\omega_1$ fix also a bijection $e_\delta\colon\delta\to\omega$. Let us assume $\diamondsuit(\mathbf{R}(\prescript{\omega}{}{\B}/\mathrm{Fin}))$ holds: in order to define a Borel function $F\colon\prescript{<\omega_1}{}{2}\to\prescript{\omega}{}{\B}$, it will be sufficient to define $F$ on a Borel subset of $\prescript{\delta}{}{2}$ for every $\omega\le\delta<\omega_1$, then extend $F$ to $\prescript{<\omega_1}{}{2}$ by assigning a constant value elsewhere.

Following the notation of the proof of \cite[Theorem 7.8]{MR2048518}, the domain of $F$ consists of pairs $\langle\vec{U},C\rangle$, where:
\begin{itemize}
\item $\vec{U}=\Seq{U_\xi}{\xi<\delta}$ for some $\omega\le\delta<\omega_1$;
\item for all $\xi<\delta$, $U_\xi\in\prescript{\omega}{}{\B}\setminus\mathrm{Fin}$;
\item for all $\xi<\eta<\delta$, $\eq{U_\eta}{\mathrm{Fin}}\le\eq{U_\xi}{\mathrm{Fin}}$;
\item $C\in\B$.
\end{itemize}
For every such $\vec{U}$, construct recursively an increasing sequence $\Seq{k_i}{i<\omega}$ of natural numbers such that $\bigwedge_{j\le i}U_{e^{-1}_\delta(j)}(k_i)>\bbot$ for all $i<\omega$, and define
\begin{equation}\label{eq:bu}
B(\vec{U})(k_i)=\bigwedge_{j\le i}U_{e^{-1}_\delta(j)}(k_i).
\end{equation}
Next, for every pair $\langle\vec{U},C\rangle$ in the domain of $F$, we let
\[
\begin{split}
F(\vec{U},C)\colon\omega&\longrightarrow\B\\
i&\longmapsto\varphi^{-1}_{B(\vec{U})(k_i)}\bigl(B(\vec{U})(k_i)\wedge C(k_i)\bigr)
\end{split}\, .
\]

By $\diamondsuit(\mathbf{R}(\prescript{\omega}{}{\B}/\mathrm{Fin}))$, there exists $g\colon\omega_1\to\prescript{\omega}{}{\B}\setminus\mathrm{Fin}$ such that for every $f\colon\omega_1\to 2$ the set $\Set{\alpha<\omega_1}{\eq{F(f\rest\alpha)}{\mathrm{Fin}}\mathrel{R}\eq{g(\alpha)}{\mathrm{Fin}}}$ is stationary. Now, using the function $g$, we construct recursively a sequence $\Seq{U_\xi}{\xi<\omega_1}$ such that:
\begin{enumerate}
\item\label{u1} for all $\xi<\omega_1$, $U_\xi\in\prescript{\omega}{}{\B}\setminus\mathrm{Fin}$;
\item\label{u2} for all $\xi<\delta<\omega_1$, $\eq{U_\delta}{\mathrm{Fin}}\le\eq{U_\xi}{\mathrm{Fin}}$.
\end{enumerate}
First of all, let $U_0\colon\omega\to\B$ be the function with constant value $\btop$ and, for each $n<\omega$, let $U_{n+1}=\langle\bbot\rangle^\frown U_n$. For the general case, suppose $\omega\le\delta<\omega$ and let $\vec{U}=\Seq{U_\xi}{\xi<\delta}$ denote the sequence constructed so far. We define $U_\delta\colon\omega\to\B$ as follows: for every $i<\omega$,
\[
U_\delta(k_i)=\varphi_{B(\vec{U})(k_i)}(g(\delta)(i)),
\]
where $B(\vec{U})(k_i)$ is defined as in \eqref{eq:bu} for the sequence $\vec{U}$. If $n\in\omega\setminus\Set{k_i}{i<\omega}$, then we let $U_\delta(n)=\bbot$. Since $g(\delta)\in\prescript{\omega}{}{\B}\setminus\mathrm{Fin}$, there exist infinitely many $i<\omega$ such that $g(\delta)(i)>\bbot$; for every such $i$, we also have $U_\delta(k_i)>\bbot$, as $\varphi_{B(\vec{U})(k_i)}$ is injective. Therefore $U_\delta\in\prescript{\omega}{}{\B}\setminus\mathrm{Fin}$ and condition~\eqref{u1} is preserved. To check condition~\eqref{u2}, take $\xi<\delta$ and observe that for all $i\ge e_\delta(\xi)$
\[
U_\delta(k_i)\le B(\vec{U})(k_i)\le U_\xi(k_i),
\]
hence $\eq{U_\delta}{\mathrm{Fin}}\le\eq{U_\xi}{\mathrm{Fin}}$. This completes the recursive construction of the sequence $\Seq{U_\xi}{\xi<\omega_1}$.

Now, we define
\[
U=\Set*{\eq{f}{\mathrm{Fin}}\in\prescript{\omega}{}{\B}/\mathrm{Fin}}{\text{there exists }\xi<\omega_1\text{ such that }\eq{U_\xi}{\mathrm{Fin}}\le\eq{f}{\mathrm{Fin}}}.
\]
It is easy to see that $U$ is a filter on $\prescript{\omega}{}{\B}/\mathrm{Fin}$; to verify that $U$ is in fact an ultrafilter, let $C\colon\omega\to\B$. Unravelling the coding, choose a function $f\colon\omega_1\to 2$ such that $f\rest\alpha=\bigl\langle\Seq{U_\xi}{\xi<\alpha},C\bigr\rangle$ for each $\alpha<\omega_1$. By stationarity, there exists some $\delta<\omega_1$ such that, denoting for simplicity $\vec{U}=\Seq{U_\xi}{\xi<\delta}$, we have
\[
\eq{g(\delta)}{\mathrm{Fin}}\le\eq[\big]{F(\vec{U},C)}{\mathrm{Fin}}\quad\text{or}\quad\eq{g(\delta)}{\mathrm{Fin}}\wedge\eq[\big]{F(\vec{U},C)}{\mathrm{Fin}}=\bbot.
\]
In case $\eq{g(\delta)}{\mathrm{Fin}}\le\eq[\big]{F(\vec{U},C)}{\mathrm{Fin}}$, then for all but finitely many $i<\omega$
\[
g(\delta)(i)\le\varphi^{-1}_{B(\vec{U})(k_i)}\bigl(B(\vec{U})(k_i)\wedge C(k_i)\bigr).
\]
Applying the isomorphism $\varphi_{B(\vec{U})(k_i)}$ to both sides of the above inequality, we get
\[
U_\delta(k_i)\le B(\vec{U})(k_i)\wedge C(k_i)
\]
and, consequently, $\eq{U_\delta}{\mathrm{Fin}}\le\eq{C}{\mathrm{Fin}}$. A completely analogous argument shows that, if $\eq{g(\delta)}{\mathrm{Fin}}\wedge\eq[\big]{F(\vec{U},C)}{\mathrm{Fin}}=\bbot$, then $\eq{U_\delta}{\mathrm{Fin}}\le\neg\eq{C}{\mathrm{Fin}}$. In conclusion, $U$ is an ultrafilter on $\prescript{\omega}{}{\B}/\mathrm{Fin}$ which, by construction, contains a cofinal subset of cardinality $\aleph_1$. It follows that $\uu(\prescript{\omega}{}{\B}/\mathrm{Fin})\le\aleph_1$.

For the reverse inequality, it is enough to use Proposition~\ref{proposition:ult} and conclude that $\aleph_1\le\uu\le\uu(\prescript{\omega}{}{\B}/\mathrm{Fin})$.
\end{proof}

We would like to show that $\C_\omega$ satisfies the hypothesis of Theorem~\ref{theorem:diamond}.

\begin{proposition}\label{proposition:hom} The Cohen algebra is Borel-homogeneous.
\end{proposition}
\begin{proof} Since the result follows from standard arguments, we only sketch the main ideas. Fix a recursive enumeration $\prescript{<\omega}{}{2}=\Set{s_n}{n<\omega}$. Let us say that $x\in\prescript{\omega}{}{2}$ is a \emph{canonical code} if:
\begin{enumerate}
\item\label{can1} for all $n,m<\omega$, if $x(n)=1$ and $s_n\subseteq s_m$, then $x(m)=1$;
\item\label{can2}for all $n<\omega$, if for every $s\in\prescript{<\omega}{}{2}$ with $s_n\subseteq s$ there exists $m<\omega$ such that $s\subseteq s_m$ and $x(m)=1$, then $x(n)$=1.
\end{enumerate}

The following claim is an immediate consequence of the definition.

\begin{claim} The set $\Set{x\in\prescript{\omega}{}{2}}{x\text{ is a canonical code}}$ is Borel.
\end{claim}

Now, for every canonical code $x$, let
\[
O_x=\bigcup\Set{N_{s_n}}{x(n)=1}
\]
be the open subset of $\prescript{\omega}{}{2}$ coded by $x$. This way, each element of $\C_\omega$ is represented by a canonical code, as in the following claim.

\begin{claim}\label{claim2} For every $X\in\mathcal{B}(\prescript{\omega}{}{2})$ there exists a canonical code $x\in\prescript{\omega}{}{2}$ such that the symmetric difference $X\bigtriangleup O_x$ is meagre.
\end{claim}
\begin{proof}[Proof of Claim~\ref{claim2}] Given $X\in\mathcal{B}(\prescript{\omega}{}{2})$, define $x\colon\omega\to 2$ by:
\[
x(n)=1\iff N_{s_n}\setminus X\text{ is meagre}.
\]
To check that $x$ is a canonical code, condition~\eqref{can1} is easy: if $N_{s_n}\setminus X$ is meagre and $s_n\subseteq s_m$, then of course $N_{s_m}\setminus X$ is meagre. Towards condition~\eqref{can2}, suppose $N_{s_n}\setminus X$ is not meagre: we shall find an extension $s\supseteq s_n$ such that for every further extension $s_m\supseteq s$ the set $N_{s_m}\setminus X$ is not meagre. By the Baire property of $N_{s_n}\setminus X$, there exists an open set $A\subseteq\prescript{\omega}{}{2}$ such that $(N_{s_n}\setminus X)\bigtriangleup A$ is meagre. Since $N_{s_n}\setminus X$ is not meagre, in particular $N_{s_n}\cap A$ cannot be empty, hence there exists $s\in\prescript{<\omega}{}{2}$ such that $s_n\subseteq s$ and $N_s\subseteq A$. It follows that $N_s\cap X$ is meagre and, therefore, for every $s_m\supseteq s$ the set $N_{s_m}\setminus X$ is not meagre, as we wanted to show.

Finally, we prove that both $O_x\setminus X$ and $X\setminus O_x$ are meagre. The first set is meagre since, by construction, it is a countable union of meagre sets. For the second, we use the Baire property of $X$ to find an open set $O\subseteq\prescript{\omega}{}{2}$ such that $X\bigtriangleup O$ is meagre. Then easily $O\subseteq O_x$, which gives that $X\setminus O_x\subseteq X\bigtriangleup O$, whence the conclusion follows.
\end{proof}

Next, we prove that the order relation on $\C_\omega$ corresponds to the pointwise order relation on canonical codes.

\begin{claim}\label{claim3} If $x$ and $y$ are canonical codes, then
\[
(\forall n<\omega)(x(n)\le y(n))\iff O_x\setminus O_y\text{ is meagre}.
\]
\end{claim}
\begin{proof}[Proof of Claim~\ref{claim3}] The left-to-right implication is clear. For the reverse implication, suppose there exists $n<\omega$ such that $x(n)=1$ but $y(n)=0$. Since $y$ is a canonical code, by condition~\eqref{can2} there exists an extension $s\supseteq s_n$ such that for every further extension $s_m\supseteq s$ we have $y(m)=0$. Combined with condition~\eqref{can1}, this implies that for all $m<\omega$, if $y(m)=1$ then $s_m$ and $s$ are incompatible. Hence $N_s\subseteq O_x\setminus O_y$, and so $O_x\setminus O_y$ is not meagre.
\end{proof}

It is also possible to obtain the incompatibility relation on $\C_\omega$, by means of the following equivalence for canonical codes $x$ and $y$:
\[
(\forall n<\omega)(x(n)=0\lor y(n)=0)\iff O_x\cap O_y\text{ is meagre}.
\]
Without extra effort, the meet operation, and in fact all Boolean operations on $\C_\omega$, can be coded as operations on canonical codes.

Finally, Borel-homogeneity follows from the observation that the Boolean algebra of clopen subsets of the Cantor space is homogeneous and, furthermore, the isomorphisms witnessing homogeneity can be taken to be Borel. By a density argument, which we leave to the reader, this extends to the whole algebra $\C_\omega$. 
\end{proof}

The main consequences for the ultrafilter number of $\C_\omega$ are summarized in the following corollary.

\begin{corollary}\leavevmode
\begin{enumerate}
\item $\cof(\M)\le\uu(\C_\omega)$;
\item $\mathbf{R}(\prescript{\omega}{}{\C_\omega}/\mathrm{Fin})$ is a Borel relational system, and $\diamondsuit(\mathbf{R}(\prescript{\omega}{}{\C_\omega}/\mathrm{Fin}))$ implies that $\uu(\C_\omega)=\aleph_1$.
\end{enumerate}
\end{corollary}
\begin{proof} The first point follows from Corollary~\ref{corollary:reap}, Remark~\ref{remark:rruu}, and Proposition~\ref{proposition:ult}. The second point follows from Proposition~\ref{proposition:ult}, Theorem~\ref{theorem:diamond}, and Proposition~\ref{proposition:hom}.
\end{proof}

From previous work \cite[Section 3]{MR4404621}, we already knew that $\uu\le\uu(\C_\omega)$ and that consistently $\uu(\C_\omega)<\non(\N)$. Therefore, the relation between $\uu(\C_\omega)$ and the cardinal invariants in Cicho{\'n}'s diagram is now completely determined.

\end{document}